\documentclass{amsart}
\title{Combinatorial Ricci curvature on cell-complex and Gauss-Bonnnet Theorem}
\author{Kazuyoshi Watanabe}
\date{}
\usepackage{amsthm,amsmath,amssymb,graphicx}
\usepackage{amscd,pifont}
\usepackage{amsrefs}

\theoremstyle{plain} 
\newtheorem{theorem}{\indent\sc Theorem}[section]
\newtheorem{lemma}[theorem]{\indent\sc lemma}
\newtheorem{corollary}[theorem]{\indent\sc corollary}
\newtheorem{proposition}[theorem]{\indent\sc proposition}

\theoremstyle{definition} 
\newtheorem{definition}[theorem]{\indent\sc Definition}

 \makeatletter
    
    \@addtoreset{equation}{section}
  \makeatother
  
  \makeatletter
\def\address#1#2{\begingroup
\noindent\parbox[t]{7.8cm}{%
\small{\scshape\ignorespaces#1}\par\vskip1ex
\noindent\small{\itshape E-mail address}%
\/: #2\par\vskip4ex}\hfill%
\endgroup}%
\makeatother

\begin{document}

\maketitle
\begin{abstract}
In this paper, we introduce a new definition of the Ricci curvature on cell-complexes and prove the Gauss-Bonnnet type theorem for graphs and 2-complexes that decompose closed surfaces. The defferential forms  on a cell complex is defined as linear maps on the chain complex, and the Laplacian operates this defferential forms. Our the Ricci curvature is defined by the combinatorial Bochner-Weitzenb\"{o}ck formula. We prove some propositionerties of combinatorial vector fields on a cell complex.
\end{abstract}
\section{Introduction}
In this paper, we  introduce a new definition of the Ricci curvature on cell-complexes and prove the Gauss-Bonnnet type theorem for graphs and 2-complexes that decompose closed surfaces. In the Riemanian geometry, the curvature plays an important role, and there are many results on the curvature on smooth manifolds. Especially the Gauss-Bonnet theorem is known as a fundamental propositionerty of a smooth closed manifold. The curvature on a cell complex was stuied in many ways. R. Forman defined the Ricci curvature on a cell complex with the Bochner-Weitzenb\"{o}ck formula on the cochain. With this curvature he also showed the Bochner's theorem, the Myers' theorem and so on. For the curvature defined by angles, Mccorollarymick Paul \cite{paul} established the Gauss-Bonnet theorem, but in the Forman's way the Gauss-Bonnet theorem does not hold.

R. Forman established the discrete Morse theory in \cite{forman-morse}. He studied the function on a cell complex and the relation between critical cells and the homology of the cell complex. He also extended this thoery to the discrete Novikov-Morse theory, and in this theory he defined a defferential form on the cell complex. This defferential form is not the cochain of the cell complex but a linear map on the chain of the cell complex. In this paper we use this defferential forms to define the Ricci curvature. We introduce the $L^2$ inner product on the space of combinatorial defferential forms, and this inner product determines the Laplacian on combinatorial defferential forms. Then the Ricci curvature is definied with the combinatorial Bochner-Weitzenb\"{o}ck formula for the combinatorial defferential forms.

For the construction of the Bochner-Weitzenb\"{o}ck formula, we need the covariant of a 1-form. For this definition we present the 0- and 2-neighbor vector. These vectors is roughly said ``the pararell vectors''. We define the covariant of a 1-form as the defference between the components of pararell vectors. Then for the cell complex with constant weights, the Ricci curvature is calculated as combinatorial computation,
\begin{eqnarray*}
\operatorname{Ric}(\omega) (\tau>\sigma) = (2- \# \{ {\rm 0 - neighbor~vector~of~} (\tau>\sigma) \}) (\omega ^\tau _\sigma)^2.
\end{eqnarray*}
This formula means that the Ricci curvature on the cell complex at the cell $\sigma$ is determined by the constructure around $\sigma$. For a graph or a 2-dimensional complex that decomposes a closed surface, the Ricci curvature for a unit vector at a vertex (resp. at a face $f$) is independent of the choice of the unit vector, and we define this value as the Gauss curvature $g_v$ (resp. $g_f$) at the vertex $v$ (resp. the face $f$). We have the following Gauss-Bonnet type theorems.
\begin{theorem}\label{main1}
Let $G$ be a finite simple graph. Then we have
\begin{eqnarray}
\sum_v g_v = 2\chi(G),
\end{eqnarray}
where the sum is taken over all vertexes $v$ and $\chi(G)$ is the Euler number of $G$.
\end{theorem}
\begin{theorem}\label{main2}
Let $M$ be a 2-dimensional quasiconvex cell complex that decomposes a 2-dimensional closed smooth surface. Then we have
\begin{eqnarray}
\sum_v g_v + \sum_f  g_f= 4\chi(M),
\end{eqnarray}
where the sums are taken over all vertexes $v$ and all faces $f$ respectivity, and $\chi(M)$ is the Euler number of $M$.
\end{theorem}

We also present the vector field on a cell complex. This is defined as a dual of a combinatorial defferential 1-form and we prove some propositionerties of this vector field.

\section{Combinatorial Defferential Forms}
\subsection{Definition of Combinatorial Defferential Forms}
In this section, we present a defferential form on a cell-complex intoruduced in \cite{forman-novikov}. Let $M$ be a regular cell complex of dimension $n$, and
\begin{eqnarray}
\begin{CD}
0 @>>> C_n(M) @>\partial>> C_{n-1}(M) @>\partial>> \cdots @>\partial>>C_0(M) @>>> 0 \nonumber
\end{CD}
\end{eqnarray}
be the real cellular chain complex of $M$. We set
\begin{eqnarray}
C_* (M) =\bigoplus_p C_p (M).
\end{eqnarray}
A linear map $\omega : C_*(M) \rightarrow C_*(M)$ is said to be of {\it degree} $d$ if for all $p=1,...,n$,
\begin{eqnarray}
\omega (C_p(M)) \subset C_{p-d} (M).
\end{eqnarray} 
We say that a linear map $\omega$ of degree $d$ is $local$ if, for each $p$ and each oriented $p$-cell $\alpha$, $\omega(\alpha)$ is a linear combination of oriented ($p-d$)-cells that are faces of $\alpha$.
\begin{definition}
For $d\geq0$, we say that a local linear map $\omega : C_*(M) \rightarrow C_*(M)$ of degree $d$ is {\it a conbinatorial defferential $d$-form}, and we denote the space of combinatorial differential $d$-forms by $\Omega ^d (M)$.
\end{definition}
We define {\it the differential of combinatorial defferential forms}
\begin{eqnarray}
d:\Omega^d(M) \rightarrow \Omega^{d+1}(M)
\end{eqnarray}
as follows. For any $\omega \in \Omega^d(M)$ and any $p$-chain $c$, we define $(d\omega) (c) \in C_{p-(d+1)} (M)$ by
\begin{eqnarray}
(d\omega)(c) = \partial (\omega(c)) - (-1)^d \omega (\partial c).
\end{eqnarray}
That is,
\begin{eqnarray}
d\omega = \partial \circ \omega - (-1)^d \omega \circ \partial.
\end{eqnarray}

\begin{lemma}[\cite{forman-novikov}]
The differential for combinatorial defferential forms satisfies the following propositionerties.
\begin{itemize}
\item $d(\Omega^d(M)) \subseteq \Omega^{(d+1)}(M)$.\\
\item $d^2=0$.
\end{itemize}
\end{lemma}
This lemma determines the differential complex
\begin{eqnarray}
\begin{CD}
\Omega^*(M) : 0 @>>> \Omega^0(M) @>d>> \Omega^1(M) @>d>> \cdots @>d>> \Omega^n(M) @>>>0. \nonumber
\end{CD}
\end{eqnarray}

\begin{theorem}[\cite{forman-novikov}]
The cohomology of this complex is isomorphic to the singular cohomology of $M$.
That is,
\begin{eqnarray}
H^* (\Omega^* (M) ) \cong H^*(M).
\end{eqnarray}
\end{theorem}

\subsection{Laplacian for Combinatorial Defferential Forms}
Let us define an inner product on $C_* (M)$. For any two $p$-cells $\sigma,\sigma'$, we set an innner product as
\begin{eqnarray}
\langle \sigma, \sigma' \rangle = \delta _{\sigma,\sigma'} w_\sigma,
\end{eqnarray}
where $\delta _{\sigma,\sigma'}$ is the Kronecker's delta, that is, $\delta _{\sigma,\sigma'}=1$ for $\sigma=\sigma'$ and the others are 0, and $w_\sigma>0$ is a weight of a cell $\sigma$. We define the $L^2$ inner product for combinatorial differential forms. For two $d$-forms $u,v$, we set
\begin{eqnarray}
\langle u,v \rangle= \sum_{\sigma} \frac{1}{w_{\sigma}} \langle u(\sigma), v(\sigma) \rangle,
\end{eqnarray}
where the sum is taken over all cells $\sigma$ in $M$. \\
Let us consider the adjecent operator of differential with respect to the inner product,
\begin{eqnarray}
d^* : \Omega^d (M) \rightarrow \Omega^{d-1}(M).
\end{eqnarray}
That is, for a $d$-form $u$ and a $(d-1)$-form $v$ we have
\begin{eqnarray}
\langle d^*u,v \rangle=\langle u,dv\rangle.
\end{eqnarray}

The space of combinatorial differential $d$-forms $\Omega ^d (M)$ is a sub vector space of the space of linear maps of degree $d$ on the chain $C_*(M)$. Then we set $p$ as the projection on the space of linear maps of degree $d$ on the chain $C_*(M)$ to the space of combinatorial differential $d$-forms $\Omega ^d (M)$. 

\begin{lemma}
For any $d$-form $\omega$, we have
\begin{eqnarray}\label{dstar}
d^* = p\circ (\partial ^* \circ \omega - (-1)^{(d-1)} \omega \circ \partial ^*).
\end{eqnarray}
\end{lemma}

\begin{proof}
 For a $p$-dimensional cell $\tau$ and a $(p-d)$-dimensional cell $\sigma$ which is a face of $\tau$, we put a $d$-form $e^{\tau} _{\sigma}$ such that 
 \begin{eqnarray}
 e^{\tau} _{\sigma} (\tau) = \sigma,
 \end{eqnarray} 
and the value is $0$ for other cells. They form a basis of $d$-forms as a real vector space.

Let $\tau,\sigma,\alpha,\beta$ be cells, and assume that $e^{\tau} _{\sigma}$ is a $d$-form and $e^{\alpha} _{\beta}$ is a $(d-1)$-form. Then we have
\begin{eqnarray}\label{dstar1}
\langle d^* e^{\tau} _{\sigma},e^{\alpha} _{\beta} \rangle &=& \langle e^{\tau} _{\sigma},d e^{\alpha} _{\beta}\rangle \\
                                                       &=& \langle e^{\tau} _{\sigma}, \partial \circ e^{\alpha} _{\beta} \rangle -(-1)^{(d-1)} \langle e^{\tau} _{\sigma},e^{\alpha} _{\beta} \circ \partial \rangle. \nonumber
\end{eqnarray}
We put $A$ as the right-hand side of equation \eqref{dstar}, and have
\begin{eqnarray}\label{dstar2}
\langle Ae^{\tau} _{\sigma},e^{\alpha} _{\beta} \rangle = \langle e^{\tau} _{\sigma},\partial e^{\alpha} _{\beta} \rangle-(-1)^{(d-1)} \langle e^{\tau} _{\sigma} \circ  \partial ^*, e^{\alpha} _{\beta}\rangle.
\end{eqnarray}
Now we calculate the last term of equation \eqref{dstar1},
\begin{eqnarray}
\langle e^{\tau} _{\sigma},e^{\alpha} _{\beta} \circ \partial \rangle &=& \sum_{c: {\rm cell}} \frac{1}{w_c} \langle e^{\tau} _{\sigma}(c) ,e^{\alpha} _{\beta} ( \partial c)\rangle\\
                                                                    &=& \frac{1}{w_c} \langle \sigma, e^\alpha _\beta (\partial \tau)\rangle \nonumber\\
&=&     \left\{ \begin{split} & \frac{w_\sigma}{w_\tau } (-1)^{\tau > \alpha} & {\rm for}~ \tau>\alpha ,~ \sigma=\beta  \\
&0 &{\rm otherwise}.
\end{split} \right. \nonumber
 \end{eqnarray}
 We calculate the last term of eqation \eqref{dstar2},
 \begin{eqnarray}
 \langle e^{\tau} _{\sigma} \circ  \partial ^*, e^{\alpha} _{\beta}\rangle &=& \sum_{c:{\rm cell}} \frac{1}{w_c} \langle e^{\tau} _{\sigma} ( \partial ^* c), e^{\alpha} _{\beta}(c)\rangle  \\
                                                                           &=& \frac{1}{w_\alpha} \langle e^\tau_\sigma (\partial ^* \alpha),\beta \rangle \nonumber \\
&=&     \left\{ \begin{split} & \frac{w_\sigma}{w_\tau } (-1)^{\tau > \alpha} & {\rm for}~ \tau>\alpha , ~\sigma=\beta  \\
&0 &{\rm otherwise}.
\end{split} \right. \nonumber
 \end{eqnarray}
Then we have
\begin{eqnarray}
d^* e^\tau _\sigma = A e^\tau _\sigma.
\end{eqnarray}
 
\end{proof}

\begin{definition}
We define {\it the Laplacian for combinatorial defferential forms} by
\begin{eqnarray}
\Delta = d d^* + d^* d.
\end{eqnarray}
\end{definition}

\begin{theorem}
Let $M$ be a finite regular cell-complex. Then we have
\begin{eqnarray}
\operatorname{Ker} (\Delta) \cong H^*(\Omega^* (M)) \cong H^*(M). 
\end{eqnarray}

\end{theorem}
\begin{proof}
The Laplacian is a self-adjoint operator on the finite dimensional vector space $\Omega ^* (M)$ and we have $\operatorname{Ker} (\Delta) = \operatorname{Ker}(d) \cap \operatorname{Ker} (d^*)$. We consider a map
\begin{eqnarray}\label{laplaceop}
\operatorname{Ker} (\Delta) &\rightarrow& H^*(\Omega^* (M))\\
 u&\mapsto& [u].\nonumber
\end{eqnarray}
This map is well-defined and injective. Next we prove that this map is surjective. The Laplacian has an eigen decomposition, and we denote the eigen values of the Laplacian by $\lambda_0,...,\lambda_k$ are eigen values. Let $u$ be a closed form. Then we have an eigen decomposition for the Laplacian
\begin{eqnarray}
\Delta u = \sum_{i=1,...,k} \lambda _i u_i,
\end{eqnarray}
where $u_i$ are eigen vectors for $\lambda_i$ respectively such that $u=\sum_{i=0,...,k} u_i$.
Then putting
\begin{eqnarray}
u' = \sum_{i=1,...,k} \frac{d^*u_i}{\lambda_i},
\end{eqnarray}
we have
\begin{eqnarray}
u = u_0 +d u'.
\end{eqnarray}
We conclude that the map \eqref{laplaceop} is an isomorphism.
\end{proof}

\subsection{Combinatorial function on cell-complex}
We realize a combinatorial 0-form as a function. We set $f\in \Omega ^0 (M)$, that is,
\begin{eqnarray}
f: C^*(M) \rightarrow C^*(M).
\end{eqnarray}
For any cell $\sigma$, we have
\begin{eqnarray}
f(\sigma) = f_\sigma \sigma,
\end{eqnarray}
and we realize $f_\sigma \in \mathbf{R}$ as the value of the function $f$. For a $p$-dimensional cell $\tau$, the derivative of $f$ is 
\begin{eqnarray}
df(\tau) = \sum_{\sigma: \tau > \sigma} (f(\tau) -f(\sigma)) (-1)^{\tau > \sigma} \sigma,
\end{eqnarray}
where  the sum is taken over all $(p-1)$-dimensional cells $\sigma$ that are faces of $\tau$, and $(-1)^{\tau > \sigma}$ is the incidence number between $\tau$ and $\sigma$.
\begin{lemma}
Let $M$ be a regular cell-complex and $f$ a function on $M$, where we identify $M$ with the set of cells of $M$. $f$ is locally constant if and only if $df=0$.
\end{lemma}

\subsection{Combinatorial 1-form on cell-complex}
Let $\omega \in \Omega^1 (M)$ be a combinatorial 1-form. For a $p$-dimensional cell $\tau$, we set
\begin{eqnarray}
\omega (\tau) = \sum_{\sigma: \tau > \sigma} \omega^\tau _\sigma(-1)^{\tau > \sigma} \sigma,
\end{eqnarray}
where  the sum is taken over all $(p-1)$-dimensional cells $\sigma$ that are faces of $\tau$, and $(-1)^{\tau > \sigma}$ is the incidence number between $\tau$ and $\sigma$. We call the pair $(\tau>\sigma)$ {\it a vector} provided that a $p$-dimensional cell $\sigma$ is a face of $(p+1)$-dimensional cell $\tau$. We say that $\omega$ has the value $\omega^\tau_\sigma$ at the vector $(\tau >\sigma)$.
For a $p$-dimensional cell $\mu$, the derivative of $\omega$ is
\begin{eqnarray}
d\omega (\mu) = \sum_{(\mu >\tau ,\tau'>\sigma)} (\omega^\mu_\tau +\omega^\tau_\sigma - \omega^\mu_{\tau'} -\omega^{\tau'}_\sigma) (-1)^{\mu>\tau} (-1)^{\tau > \sigma} \sigma,
\end{eqnarray}
where the sum is taken over all two $(p-1)$-dimensional cells $\tau,\tau'$ and $(p-2)$-dimensional cells $\sigma$ such that 
\begin{eqnarray}
\mu > \tau > \sigma,~ \mu > \tau' > \sigma,~ \tau \neq \tau'.
\end{eqnarray}

\begin{proposition}
For a combinatorial 1-form $\omega$, we have $d\omega =0$ if and only if
\begin{eqnarray}
\omega^\mu_\tau +\omega^\tau_\sigma =\omega^\mu_{\tau'} +\omega^{\tau'}_\sigma
\end{eqnarray}
 for any $p$-dimensional cell $\mu$, any two $(p-1)$ dimensional cells $\tau,\tau'$ and any $(p-2)$-dimensional cell $\sigma$ such that 
\begin{eqnarray}
\mu > \tau > \sigma, \mu > \tau' > \sigma, \tau \neq \tau'.
\end{eqnarray}
\end{proposition}

For any cell $\sigma$, the dual derrivative of $\omega$ is
\begin{eqnarray}
d^* \omega (\sigma) =\left( -\sum_{\tau :\tau>\sigma} \frac{w_\sigma}{w_\tau} \omega^\tau_\sigma + \sum_{\rho:\rho<\sigma}\frac{w_\sigma}{w_\rho} \omega^\sigma_\rho  \right) \sigma,
\end{eqnarray}
where the first sum is taken over all $(p+1)$-dimensional cells $\tau$ that have $\sigma$ as a face, and the second sum is over all $(p-1)$-dimensional cells $\rho$ that are the faces of $\sigma$.

\section{Combinatorial Vector field on cell-complex}
\begin{definition}
We call a linear map $X:C_* (M)\rightarrow C_{(*+1)} (M)$ {\it a combinatorial vector field} on a cell-complex $M$ provided that for any cell $\sigma$ any component of $X(\sigma)$ has $\sigma$ as a face. 
\end{definition}
In the same manner as a 1-form, for a $p$-dimensional cell $\sigma$ we set
\begin{eqnarray}
X(\sigma)= \sum_{\tau:\tau>\sigma} X^\tau_{\sigma} (-1)^{\tau>\sigma} \tau,
\end{eqnarray}
where  the sum is taken over all $(p-1)$-dimensional cells $\tau$ that have $\sigma$ as a face, and $(-1)^{\tau > \sigma}$ is the incidence number between $\tau$ and $\sigma$.

For a 1-form $\omega$ and a vector field $X$ on $M$, we define {\it the pairing}
\begin{eqnarray}
\omega (X) (\sigma) = \sum_{\tau:\tau>\sigma} \omega^\tau_\sigma X^\tau_\sigma.
\end{eqnarray}
Then for a function $f$ on $M$, we define
\begin{eqnarray}
X(f) (\sigma) = df(X) (\sigma) = \sum_{\tau:\tau>\sigma} X^\tau_\sigma (f(\tau) -f(\sigma)).
\end{eqnarray}
\begin{definition}
Let $f$ be a function on $M$. We define {\it the gradient vector field $\operatorname{grad}(f)$ of $f$} by
\begin{eqnarray}
\operatorname{grad}(f)^\tau_{\sigma} = \frac{w_\sigma}{w_\tau} (f(\tau)-f(\sigma)).
\end{eqnarray}
Let $X$ be a vector field on $M$. We also define {\it the divergence $\operatorname{div} (X)$  of $f$} by
\begin{eqnarray}
\operatorname{div} (X) (\sigma) = -\sum_{\tau^{(p+1)}:\tau>\sigma} X^\tau_\sigma +\sum_{\rho^{(p-1)}:\rho>\sigma } X^\sigma_\rho
\end{eqnarray}
\end{definition}
We define the inner product for vector fields in the same manner as combinatorial defferential forms, i.e. for two vector fields $X,Y$
\begin{eqnarray}
\langle X,Y\rangle(\sigma) = \frac{1}{w_\sigma} \langle X(\sigma),Y(\sigma) \rangle=\sum_{\tau:\tau>\sigma} \frac{w_\tau}{w_\sigma} X^\tau_\sigma Y^\tau_\sigma.
\end{eqnarray}
Then we have
\begin{eqnarray}
df(X) = \langle X,\operatorname{grad}(f)\rangle.
\end{eqnarray}

\begin{definition}
For a function $f$ on $M$, we define {\it the integral of $f$ over $M$} by
\begin{eqnarray}
\int_M f = \sum_{\sigma} f(\sigma),
\end{eqnarray}
where the sum is taken over all cells of $M$.
\end{definition}
\begin{theorem}\label{Green}
We assume that $M$ is a finite regular cell-complex. Let $f$ be a function on $M$ and $X$ a vector fieldon $M$. Then we have
\begin{eqnarray}
\int_M \langle \operatorname{grad}(f),X\rangle =\int_M f\operatorname{div}(X).
\end{eqnarray}
\end{theorem}

\begin{proof}
For a cell $\sigma$, we have
\begin{eqnarray}
\sum_{\sigma} \langle\operatorname{grad}(f),X\rangle (\sigma)&=&\sum_{\sigma} \sum_{\tau:\tau>\sigma} \frac{w_\tau}{w_\sigma} X^\tau_\sigma \cdot \frac{w_\sigma}{w_\tau}(f(\tau)-f(\sigma))\\
&=&\sum_{(\tau>\sigma)} X^\tau_\sigma (f(\tau)-f(\sigma))\nonumber \\
&=&\sum_{\sigma} f(\sigma)  \left( -\sum_{\tau^{(p+1)}:\tau>\sigma} X^\tau_\sigma +\sum_{\rho^{(p-1)}:\rho>\sigma } X^\sigma_\rho \right)\nonumber \\
&=&\int_M f \operatorname{div}(X).\nonumber
\end{eqnarray}
\end{proof}

\begin{corollary}
For any vector field $X$ on $M$ we have
\begin{eqnarray}
\int_M \operatorname{div}(X) =0.
\end{eqnarray}

\begin{proof}
For a constant function $f$, the gradient of $f$ vanishes. Then we take a constant function $f$ as
\begin{eqnarray}
f(\sigma)=1
\end{eqnarray}
for any cell $\sigma$. Then we have the corollary from Theorem \ref{Green}.
\end{proof}

\end{corollary}
\begin{corollary}
For any function $f$ on $M$ we have
\begin{eqnarray}
\int_M \Delta f =0.
\end{eqnarray}
\end{corollary}
\begin{proof}
For any cell $\sigma$,
\begin{eqnarray}
\Delta f (\sigma)&=&  \left( -\sum_{\tau^{(p+1)}:\tau>\sigma} (f(\tau)-f(\sigma)) +\sum_{\rho^{(p-1)}:\rho>\sigma } (f(\sigma)-f(\rho)) \right)\\
&=& \operatorname{div} (\operatorname{grad}(f)) (\sigma).\nonumber
\end{eqnarray}
Then we have the corollaryorally from the previous corollary.
\end{proof}

\section{Combinatorial Ricci curvature}
\subsection{Combinatorial Ricci curvature}
\begin{definition}
Let $M$ be a regular cell complex. We say that $M$ is {\it quasiconvex} if for every two distinct $(p+1)$-cells $\tau_1$ and $\tau_2$ of $M$, if $\bar{\tau_1} \cap \bar{\tau_2}$ contains a $p$-cell $\sigma$, then $\bar{\tau_1} \cap \bar{\tau_2} =\bar{\sigma}$. In particular this implies that $\bar{\tau_1} \cap \bar{\tau_2}$ contains at most one $p$-cell.
\end{definition}
Let $M$ be a regular quasiconvex cell complex.
\begin{definition}
Let $\tau, \sigma$ be two cells of $M$ such that the dimension is $(p+1)$ and $p$ respectivity and $\sigma$ is a face of $\tau$. 

We define {\it a 0-neighbor vector of $(\tau>\sigma)$} as the following.
\begin{itemize}
\item The vectors $(\tau' > \sigma)$ for $(p+1)$-cells $\tau' \neq \tau$ such that there are no $(p+2)$-cell $\mu$ such that $\mu>\tau, \tau'$.
\item The vectors $(\tau>\sigma')$ for $p$-cells $\sigma' \neq \sigma$ such that there are no $(p-1)$-cell $\rho$ such that $\sigma,\sigma'>\rho$.
\end{itemize}
 
 We define {\it a 2-neighbor vector of $(\tau>\sigma)$} as the following.
\begin{itemize}
\item The vectors $(\mu > \tau')$ for $(p+1),(p+2)$-cells $\tau'$ and $\mu$ such that $\mu > \tau>\sigma$, $\mu>\tau'>\sigma$ and $\tau \neq \tau'$. 
\item The vectors $(\sigma'>\rho)$ for $(p-1),p$-cells $\rho$ and $\sigma'$ such that $\tau>\sigma>\rho$, $\tau>\sigma'>\rho$ and $\sigma \neq \sigma'$.
\end{itemize}
\end{definition}

\begin{definition}
For a combinatorial 1-form $\omega$ on $M$, we define {\it the combinatorial covariant derivative} as
\begin{eqnarray}
|\nabla \omega|^2 (\tau>\sigma) &=& \sum_{(\mu>\tau') ; {\rm 2-neighbor}} \frac{w_\sigma}{w_\mu} (\omega ^\tau _\sigma - \omega^\mu_{\tau'})^2 + \sum_{(\sigma' >\rho); {\rm 2-neighbor } } \frac{w_\rho}{\tau} (\omega^\tau_\sigma - \omega^{\sigma'}_\rho)^2 \nonumber\\
                                               &+& \sum_{(\tau' >\sigma) ; {\rm 0-neighbor}}\frac{(w_\sigma)^2}{w_\tau w_{\tau'}} (\omega^\tau_\sigma +\omega^{\tau'}_\sigma)^2 + \sum_{(\tau >\sigma') ; {\rm 0-neighbor}} \frac{w_\sigma w_{\sigma'}}{(w_\tau)^2} (\omega^\tau_\sigma +\omega^\tau_{\sigma'})^2 \nonumber,
\end{eqnarray}
where the sums are taken over all 2-neighbor vectors and 0-neighbor vectors for $(\tau>\sigma)$ respectively. 
\end{definition}

\begin{definition}
For a combinatorial 1-form $\omega$ on $M$, we define {\it the Laplacian of $|\omega|^2$ } as
\begin{eqnarray}
\Delta^\flat |\omega|^2 (\tau>\sigma)  = \sum_{(\mu>\tau') ; {\rm 2-neighbor}} \frac{w_\sigma}{w_\mu} ((\omega ^\tau _\sigma)^2 - (\omega^\mu_{\tau'})^2) + \sum_{(\sigma' >\rho); {\rm 2-neighbor} } \frac{w_\rho}{\tau} ((\omega^\tau_\sigma)^2 - (\omega^{\sigma'}_\rho)^2) \nonumber\\
                                               + \sum_{(\tau' >\sigma) ; {\rm 0-neighbor}} \frac{(w_\sigma)^2}{w_\tau w_{\tau'}}((\omega^\tau_\sigma)^2 +(\omega^{\tau'}_\sigma)^2) + \sum_{(\tau >\sigma') ; {\rm 0-neighbor }}  \frac{w_\sigma w_{\sigma'}}{(w_\tau)^2} ((\omega^\tau_\sigma)^2 +(\omega^\tau_{\sigma'})^2), \nonumber
\end{eqnarray}
where the sums are taken over all 2-neighbor vectors and 0-neighbor vectors for $(\tau>\sigma)$ respectively. 
\end{definition}

This Laplacian is symmetry for vectors, hence we have
\begin{eqnarray}
\sum_{(\tau>\sigma)} \Delta^\flat |\omega|^2 (\tau>\sigma) =0,
\end{eqnarray}
where the sum is taken over all vectors.
\begin{definition}
For a combinatorial 1-form $\omega$, we define {\it the Ricci curvature on a vector $(\tau>\sigma)$} as
\begin{eqnarray}
\operatorname{Ric} (\omega) (\tau>\sigma) = \langle \Delta \omega, \omega \rangle (\tau>\sigma) -\frac{1}{2} |\nabla \omega|^2  (\tau>\sigma)+ \frac{1}{2} \Delta^\flat |\omega|^2(\tau>\sigma).\nonumber
\end{eqnarray}
\end{definition}
\begin{lemma}
For any combinatorial 1-form $\omega$ on $M$, we have
\begin{eqnarray}
\langle \Delta \omega, \omega\rangle (\tau>\sigma) = - \sum_{(\mu>\tau') ; {\rm 2-neighbor}} \frac{w_\sigma}{w_\mu}\omega ^\tau _\sigma \omega^\mu_{\tau'} - \sum_{(\sigma' >\rho); {\rm 2-neighbor} }\frac{w_\rho}{w_\tau} \omega^\tau_\sigma \omega^{\sigma'}_\rho \nonumber\\
                                               + \sum_{(\tau' >\sigma) ; {\rm 0-neighbor}} \frac{(w_\sigma)^2}{w_\tau w_{\tau'}} \omega^\tau_\sigma \omega^{\tau'}_\sigma + \sum_{(\tau >\sigma') ; {\rm 0-neighbor }} \frac{w_\sigma w_{\sigma'}}{(w_\tau)^2} \omega^\tau_\sigma \omega^\tau_{\sigma'} \nonumber\\ 
                                               +\sum_{(\mu>\tau') ; {\rm 2-neighbor}}\left( \frac{(w_\sigma)^2}{w_\tau w_{\tau'}} - \frac{w_\sigma}{w_\mu} \right)  \omega^\tau_\sigma \omega^{\tau'}_\sigma +\sum_{(\sigma' >\rho); {\rm 2-neighbor} }  \left(\frac{w_\sigma w_{\sigma'}}{(w_\tau)^2} - \frac{w_\rho}{w_\tau} \right)  \omega^\tau_\sigma \omega^\tau_{\sigma'}\nonumber  \\
                                               + (\# \{ 2-{\rm neighbor~vector}\} +2) \left( \frac{w_\sigma}{w_\tau} \right) ^2 (\omega^\tau_\sigma)^2.\nonumber
\end{eqnarray}
\end{lemma}
\begin{proof}
For a 1-form $\omega$, we have
\begin{eqnarray}
(d^*d\omega )^\tau _\sigma &=&  \sum_{(\mu>\tau') ; {\rm 2-neighbor}} \frac{w_\mu}{w_\tau} (\omega ^\mu _\tau +\omega^\tau_\sigma- \omega^\mu_{\tau'}-\omega^{\tau'} _\sigma ) +\\
 &&\sum_{(\sigma' >\rho); {\rm 2-neighbor } }\frac{w_\sigma}{w_\rho} (\omega^\tau_\sigma +\omega^\sigma_\rho - \omega^\tau_ {\sigma'} - \omega^{\sigma'}_\rho ) \nonumber\\
 (dd^*\omega) ^\tau_\sigma &=& -\sum_{\mu>\tau} \frac{w_\mu}{w_\tau} \omega^\mu_\tau +\sum_{\tau>\sigma'} \frac{w_\tau}{w_{\sigma'}} \omega^\tau_{\sigma'} + \sum_{\tau'>\sigma} \frac{w_{\tau'}}{w_\sigma} \omega^{\tau'}_{\sigma} - \sum_{\sigma>\rho} \frac{\sigma}{\rho} \omega^\sigma_\rho.\nonumber
\end{eqnarray}
Then the Laplacian of $\omega$ is
\begin{eqnarray}
(\Delta \omega)^\tau_\sigma =-\sum_{(\mu>\tau') ; {\rm 2-neighbor}} \frac{w_\mu}{w_\tau} \omega^\mu_{\tau'} - \sum_{(\sigma' >\rho); {\rm 2-neighbor } } \frac{w_\sigma}{w_\rho} \omega^{\sigma'}_\rho \nonumber\\
 +\sum_{(\tau' >\sigma) ; {\rm 0-neighbor}} \frac{w_{\tau'}}{w_\sigma} \omega^{\tau'}_\sigma + \sum_{(\tau >\sigma') ; {\rm 0-neighbor}}  \frac{w_\tau}{w_{\sigma'}} \omega^\tau_{\sigma'} \nonumber \\
  +\sum_{(\mu>\tau') ; {\rm 2-neighbor}}\left( \frac{w_\sigma}{w_{\tau'}} - \frac{w_\tau}{w_\mu} \right) \omega^{\tau'}_\sigma +\sum_{(\sigma' >\rho); {\rm 2-neighbor} }  \left(\frac{ w_{\sigma'}}{w_\tau} - \frac{w_\rho}{w_\sigma} \right) \omega^\tau_{\sigma'}\nonumber  \\
                                               + (\# \{ 2-{\rm neighbor~vector}\} +2)  \frac{w_\sigma}{w_\tau}(\omega^\tau_\sigma).\nonumber
\end{eqnarray}
Takeing the innner product of $\omega$ and $\Delta \omega$, we have the lemma.
\end{proof}

\begin{theorem}
Let $M$ be a regular quasiconvex cell-complex, and $(\tau>\sigma)$ a vector on $M$. For a combinatorial 1-form $\omega$ on $M$, the Ricci curvature $\operatorname{Ric}(\omega)$ is reprenseted by
\begin{eqnarray}
\operatorname{Ric}(\omega) (\tau>\sigma) = (2- \# \{ {\rm 0 - neighbor~vector~of~} (\tau>\sigma) \})  \left( \frac{w_\sigma}{w_\tau} \right) ^2 (\omega ^\tau _\sigma)^2\label{wricci}\\
+\sum_{(\mu>\tau') ; {\rm 2-neighbor}}\left( \frac{(w_\sigma)^2}{w_\tau w_{\tau'}} - \frac{w_\sigma}{w_\mu} \right)  \omega^\tau_\sigma \omega^{\tau'}_\sigma +\sum_{(\sigma' >\rho); {\rm 2-neighbor} }  \left(\frac{w_\sigma w_{\sigma'}}{(w_\tau)^2} - \frac{w_\rho}{w_\tau} \right)  \omega^\tau_\sigma \omega^\tau_{\sigma'}.\nonumber 
\end{eqnarray}
In particular, with the assumption that the weight of each cell is constant, we have
\begin{eqnarray}
\operatorname{Ric}(\omega) (\tau>\sigma) = (2- \# \{ {\rm 0 - neighbor~vector~of~} (\tau>\sigma) \})  \left( \frac{w_\sigma}{w_\tau} \right) ^2 (\omega ^\tau _\sigma)^2\nonumber. \label{ricci}
\end{eqnarray}
\end{theorem}
\begin{proof}
With the above lemma, we have
\begin{eqnarray*}
&&2\langle \Delta \omega, \omega\rangle (\tau>\sigma) \\
&=& - \sum_{(\mu>\tau') ; {\rm 2-neighbor}} \frac{w_\sigma}{w_\mu}( ( \omega ^\tau _\sigma- \omega^\mu_{\tau'})^2 -  ((\omega ^\tau _\sigma)^2 - (\omega^\mu_{\tau'})^2) -2(\omega^\tau_\sigma)^2 ) \\
&&- \sum_{(\sigma' >\rho); {\rm 2-neighbor } } \frac{w_\rho}{w_\tau}((\omega^\tau_\sigma - \omega^{\sigma'}_\rho)^2 - ((\omega^\tau_\sigma)^2 - (\omega^{\sigma'}_\rho)^2) -2(\omega^\tau_\sigma)^2)\\
   &&                                            + \sum_{(\tau' >\sigma) ; {\rm 0-neighbor}}\frac{(w_\sigma)^2}{w_\tau w_{\tau'}} (( \omega^\tau_\sigma+ \omega^{\tau'}_\sigma)^2- ((\omega^\tau_\sigma)^2 +(\omega^{\tau'}_\sigma)^2)-2(\omega^\tau_\sigma)^2)\\
   &&     + \sum_{(\tau >\sigma') ; {\rm 0-neighbor}} \frac{w_\sigma w_{\sigma'}}{(w_\tau)^2} ((\omega^\tau_\sigma +\omega^\tau_{\sigma'})^2 - ((\omega^\tau_\sigma)^2 +(\omega^\tau_{\sigma'})^2)-(\omega^\tau_\sigma )^2)\\
         &&                                      + 2 (\# \{ 2-{\rm neighbor~vector}\} +2) (\omega^\tau_\sigma)^2\\
&&+2\sum_{(\mu>\tau') ; {\rm 2-neighbor}}\left( \frac{(w_\sigma)^2}{w_\tau w_{\tau'}} - \frac{w_\sigma}{w_\mu} \right)  \omega^\tau_\sigma \omega^{\tau'}_\sigma+2\sum_{(\sigma' >\rho); {\rm 2-neighbor} }  \left(\frac{w_\sigma w_{\sigma'}}{(w_\tau)^2} - \frac{w_\rho}{w_\tau} \right)  \omega^\tau_\sigma \omega^\tau_{\sigma'}  \\
&=& |\nabla \omega|^2 (\tau>\sigma) - \Delta^\flat |\omega|^2 (\tau>\sigma) +2(2- \# \{ {\rm 0-neighbor~vector} \}) (\omega ^\tau _\sigma)^2.\\
&&+2\sum_{(\mu>\tau') ; {\rm 2-neighbor}}\left( \frac{(w_\sigma)^2}{w_\tau w_{\tau'}} - \frac{w_\sigma}{w_\mu} \right)  \omega^\tau_\sigma \omega^{\tau'}_\sigma +2\sum_{(\sigma' >\rho); {\rm 2-neighbor} }  \left(\frac{w_\sigma w_{\sigma'}}{(w_\tau)^2} - \frac{w_\rho}{w_\tau} \right)  \omega^\tau_\sigma \omega^\tau_{\sigma'}. 
\end{eqnarray*}
Then we have the equation \eqref{wricci}.
\end{proof}

\section{Combinatorial Gauss-Bonnet Theorem}
\subsection{Gauss-Bonnet Theorem for Graph}
Let $G=(V,E)$ be a finite simple graph, where $V$ is the set of vertexes and $E$ the set of edges. We realize $G$ as $1$-dimensional cell complex, i.e. vertexes are $0$-cells and edges are $1$-cells.
\begin{lemma}
Let $v$ and $e$ be a vertex and an edge of $G$ respectivity such that $e>v$. We take a combinatorial $1$-form $\omega$ on $G$.
Then we have
\begin{eqnarray}
\operatorname{Ric}(\omega) (e>v) = (2- \operatorname{deg}(v))  \left( \frac{w_v}{w_e} \right) ^2 (\omega ^e _v)^2,
\end{eqnarray}
where $\operatorname{deg}(v)$ is the degree of $v$.
\end{lemma}

\begin{proof} $2$-cell $f$ such that $f>e, e'$, and there are no $(-1)$-cell $\rho$ such that $v,v'>\rho$ respectively. 
Let $v$ and $e$ be a vertex and an edge of $G$ respectivity such that $e>v$. For the definition of 0-neighbor vector, we find two vectors $(e' >v)$ and $(e>v')$ such that there are no

Since there are no $2$-cells in $G$, the vector $(e'>v)$ is a 0-neighbor vector for any $e'$ which has the vertex $v$ except for the edge $e$. Hence there are exactly $\operatorname{deg}(v)-1$ edges that satisfy the above condition. Since there are no $(-1)$-cells, there is only one vertex $v'$ such that $(e>v')$ is a 0-neighbor vector for $(e>v)$. Then the number of 0-neighbor vectors for $(e>v)$ is $\operatorname{deg}(v)$. 

With the definition of a 2-neighbor vecotor, there are not 2-neighbor vectors for the vector $(e>v)$. We have the lemma from the eqation \eqref{wricci}.
 \end{proof}

With this lemma, we immediately have the following lemma.
\begin{lemma}
We take a combinatorial $1$-form $\omega$ on $G$ such that for any vertex $v$
\begin{eqnarray}
\sum_{e;e>v} \left( \frac{w_v}{w_e} \right) ^2 ( \omega^e_v)^2 =1,
\end{eqnarray}
where the sum is taken over all edges $e$ such that $e>v$. Then for any vertex $v$ we have
\begin{eqnarray}
\sum_{e; e>v} \operatorname{Ric}(\omega) (e>v) = 2- \operatorname{deg}(v).
\end{eqnarray}
\end{lemma}
For a smooth surface the Gauss curvature at a point $p$ equal to the Ricci curvature for a unit vector at $p$. The following definition is an analogue to this fact.
\begin{definition}
We define {\it the Gauss curvature for a vertex $v$} by
\begin{eqnarray}
g_v = 2- \operatorname{deg}(v).
\end{eqnarray}
\end{definition}

\begin{proof}[Proof of the Theorem \ref{main1}.]
From the definition of the Gauss curvature, we have
\begin{eqnarray}
\sum_v g_v &=& \sum_v (2- \operatorname{deg}(v))\\
                &=& 2 \# V - \sum_v \operatorname{deg}(v)\nonumber \\
                &=&  2 \# V - 2 \# E \nonumber \\
                &=& 2 \chi(G). \nonumber
\end{eqnarray}
\end{proof}

\subsection{Gauss-Bonnet Theorem for 2-complex}
Let $M$ be a 2-dimensional quasiconvex cell complex that decomposes a 2-dimensional closed smooth surface. 

\begin{lemma}
Let $v$ and $e$ be a vertex and an edge on $M$ respectivity such that $e>v$. We take a combinatorial $1$-form $\omega$ on $M$.
Then we have
\begin{eqnarray}
\operatorname{Ric}(\omega) (e>v) = (4- \operatorname{deg}(v))  \left( \frac{w_v}{w_e} \right) ^2  (\omega ^e _v)^2 +\sum_{(f>e') ; {\rm 2-neighbor}}\left( \frac{(w_v)^2}{w_e w_{e'}} - \frac{w_v}{w_f} \right)  \omega^v_e \omega^v_{e'} \nonumber,
\end{eqnarray}
where $\operatorname{deg}(v)$ is the degree of $v$ and the sum is taken over all 2-neighbor vectors of the vector $(e>v)$.
\end{lemma}

\begin{proof}
Let $v$ and $e$ be a vertex and an edge on $M$ respectivity such that $e>v$. For the definition of a 0-neghbor vecotor, we find the vectors $(e' >v)$ and $(e>v')$ such that there are no $2$-cell $f$ such that $f>e, e'$ and there are no $(-1)$-cell $\rho$ such that $v,v'>\rho$ respectively.

For the edge $e$ there are exactly two faces that have the edge $e$. For exactly two edges $e'$ the vectors $(e'>v)$ are not 0-neghbor vecotors of $(e>v)$. The number of 0-neghbor vecotors of $(e>v)$ is $\operatorname{deg}(v)-3$. Since there are no $(-1)$-cells in $M$, there is only one vertex $v'$ such that $(e>v')$ is a 0-neighbor vector for $(e>v)$. Then the number of 0-neighbor vecotors for $(e>v)$ is $\operatorname{deg}(v)-2$. We have the lemma from the equation \eqref{ricci}.
 \end{proof}

For a vertex $v$ we consider the sum
\begin{eqnarray}
\operatorname{Ric}(\omega)(v) :=\sum_{e;e>v} \operatorname{Ric}(\omega) (e>v),
\end{eqnarray}
where the sum is taken over all edges $e$ that have the vertex $v$. This is a quadratic form for real basises $\{\frac{w_v}{w_e}  \omega ^e _v  \}_{e>v}$. The trace with this basises is  $\operatorname{deg}(v) (4- \operatorname{deg}(v))$.
For a smooth manifold the scalar curvature is a trace of the Ricci curvature. The next definition is an analogue to this fact.
\begin{definition}
We define {\it the scalar curvature $S(v)$ at a vertex $v$} as
\begin{eqnarray}
S(v) = \operatorname{trace}  \operatorname{Ric}(\omega)(v) = \operatorname{deg}(v) (4- \operatorname{deg}(v)).
\end{eqnarray}
\end{definition}

For a smooth surface the scalar curvature is the twice of the Gauss curvature. The next definition is analogue to this fact.

\begin{definition}
We define {\it the Gauss curvature at a vetex $v$} as
\begin{eqnarray}
g_v = \frac{S(v)}{\operatorname{deg}(v)} = 4- \operatorname{deg}(v). 
\end{eqnarray}
\end{definition}

\begin{lemma}
Let $e$ and $f$ be an edge and a face of $M$ respectivity such that $f>e$. We take a combinatorial $1$-form $\omega$ on $M$.
Then we have
\begin{eqnarray}
\operatorname{Ric}(\omega) (f>e) = (4- \operatorname{deg}(f))  \left( \frac{w_e}{w_f} \right) ^2  (\omega ^f _e)^2 +\sum_{(e'>v) ; {\rm 2-neighbor}}\left( \frac{w_e w_{e'}}{(w_f)^2} - \frac{w_v}{w_f} \right)  \omega^e_f \omega^{e'}_f \nonumber,
\end{eqnarray}
where $\operatorname{deg}(f)$ is the degree of $f$, that is, the number of edges of $f$ and the sum is taken over the all 2-neighbor vector of the vector $(f>e)$.
\end{lemma}

\begin{proof}
Let $e$ and $f$ be an edge and a face of $M$ respectivity such that $f>e$. For the definition of a 0-neghbor vecotor, we find two vectors $(f' >e)$ and $(f>e')$ such that there are no $3$-cell $\sigma$ such that $\sigma>f, f'$ and there are no $0$-cell $v$ such that $e,e'>v$ respectively.

For the edge $e$ there are exactly two faces that have $e$ as an edge. Then for only one face $f'$ the vector $(f'>e)$ is a 0-neighbor vector of $(f>e)$. For edges of the face $f$, exactly two edges $e_1,e_2$ intersect with the edge $e$, then the two vectors $(f>e_1)$ and $(f>e_2)$ are not 0-neighbor vectors of $(f>e)$. For the other edges $e'$ of the face $f$, the vector $(f>e')$ is a 0-neighbor vector of $(f>e)$. Then the number of 0-neighbor vecotors for $(e>v)$ is $\operatorname{deg}(f)-2$. We have the lemma from the eqation \eqref{ricci}.
 \end{proof}

For a face $f$ we consider the next sum,
\begin{eqnarray}
\operatorname{Ric}(\omega)(f) :=\sum_{e;f>e} \operatorname{Ric}(\omega) (f>e),
\end{eqnarray}
where the sum is taken over the all edges $e$ contained in the boundary of the face $f$.

\begin{definition}
We define {\it the scalar curvature $S(f)$ at a face $f$} as
\begin{eqnarray}
S(f) = \operatorname{trace}  \operatorname{Ric}(\omega)(f) = \operatorname{deg}(f) (4- \operatorname{deg}(f)).
\end{eqnarray}

We define {\it the Gauss curvature at a face $f$} as
\begin{eqnarray}
g_f = \frac{S(f)}{\operatorname{deg}(f)} = 4- \operatorname{deg}(f). 
\end{eqnarray}
\end{definition}

If all weights of cells of $M$ are constants, we conclude the following lemma that is an analogue to the smooth surface.
\begin{lemma}
Let $M$ be a 2-dimensional quasiconvex cell complex that decomposes a 2-dimensional closed smooth surface. We assume that all weights of cells of $M$ are constants.
\begin{enumerate}
\item Let $v$ be a vertex of $M$. We take a combinatorial $1$-form $\omega$ on $M$ such that
\begin{eqnarray}
\sum_{e;e>v}( \omega^e_v)^2 =1,
\end{eqnarray}
where the sum is taken over all edges $e$ such that $e>v$. Then we have
\begin{eqnarray}
\sum_{e; e>v} \operatorname{Ric}(\omega) (e>v) = 4- \operatorname{deg}(v)=g_v.
\end{eqnarray}
\item 
Let $f$ be a face of $M$. We take a combinatorial $1$-form $\omega$ on $M$ such that
\begin{eqnarray}
\sum_{e;f>e}( \omega^f_e)^2 =1,
\end{eqnarray}
where the sum is taken over all edges $e$ such that $f>e$. Then we have
\begin{eqnarray}
\sum_{e; f>e} \operatorname{Ric}(\omega) (f>e) = 4- \operatorname{deg}(f)=g_f.
\end{eqnarray}
\end{enumerate}
\end{lemma}

\begin{proof}[Proof of Theorem \ref{main2}.]
We denote $V,E$ and $F$ by the numbers of vertexes, edges and faces in $M$ respectivity.\\
From the definition of the Gauss curvature, we have
\begin{eqnarray}
\sum_v g_v +\sum_f g_f &=& \sum_v (4- \operatorname{deg}(v)) +\sum_f (4- \operatorname{deg}(f)) \\
                &=& 4  V - \sum_v \operatorname{deg}(v)  +4 F -    \sum_f  \operatorname{deg}(f)    \nonumber \\
                &=&  4  V - 2  E + 4F -2 E \nonumber \\
                &=& 4(V-E+F)            \nonumber \\
                &=& 4\chi(M).\nonumber
\end{eqnarray}

\end{proof}

 \bigskip
\address{Kazuyoshi Watanabe\\
Mathematical Institute \\
Tohoku University \\
Sendai 980-8578 \\
Japan
}
{kazuyoshi.watanabe.q5@dc.tohoku.ac.jp
}

\end{document}